\documentclass[a4paper, reqno, 11pt]{amsart}
\usepackage{color}
\makeatletter

\@addtoreset{equation}{section}
\makeatother
\usepackage{setspace}
\usepackage{xcolor}
\usepackage{here}
\usepackage{graphicx}
\usepackage{tikz}
\usepackage{placeins}
\usepackage{float}
\usepackage{hyperref}
\usepackage[T1]{fontenc}
\everymath{\displaystyle}
\usepackage{amsmath}
\usepackage{amssymb}
\usepackage{latexsym}
\usepackage{amsthm}
\usepackage{hyperref}
\usepackage{mathtools}
\usepackage{mathrsfs}
\newtheorem{Thm}{Theorem}[section]
\newtheorem{Que}[Thm]{Question}
\newtheorem{Lem}[Thm]{Lemma}
\newtheorem{Prop}[Thm]{Proposition}

\newtheorem{Cor}[Thm]{Corollary}

\theoremstyle{definition}

\newtheorem{Def}[Thm]{Definition}
\newtheorem{Rem}[Thm]{Remark}

\begin{document}

\title[]{The two-dimensional Matkowski--Sut\^o equation with holomorphic and strictly increasing generators}
\author{Kazuki Okamura}
\date{\today}
\address{Department of Mathematics, Faculty of Science, Shizuoka University, 836 Ohya, Suruga-ku, Shizuoka 422-8529, Japan}
\email{okamura.kazuki@shizuoka.ac.jp}

\begin{abstract}
We study the two-dimensional Matkowski--Sut\^o equation, which asks for two quasi-arithmetic means whose sum is twice the arithmetic mean, in two settings.
For holomorphic injective generators with convex images on a convex domain in the complex plane, the solutions are exactly the affine pairs and the exponential pairs with a nonzero complex exponent, up to affine changes of the generators. 
The admissible exponents depend on the shape of the domain and are described by a curvature criterion for its boundary. 
In the monotone-operator framework of T\'oth, we construct an infinite-dimensional family of non-affine shear pairs on the whole plane. 
Their generators are strictly increasing in the sense of monotone operators and need not be differentiable. 
These pairs solve the weighted equation for any number of variables. 
The rigidity of the one-dimensional problem, due to Dar\'oczy and P\'ales, persists under holomorphy but not under monotonicity.
\end{abstract}

\maketitle

\section{Introduction}

A quasi-arithmetic mean on an interval $I \subset \mathbb{R}$ is a two-variable mean of the form $A_{\varphi}(x,y) = \varphi^{-1}\left( \tfrac{1}{2}(\varphi(x) + \varphi(y)) \right)$, where the generator $\varphi$ is a continuous strictly monotone function on $I$.
These means were characterized axiomatically by Kolmogorov \cite{Kolmogorov}, Nagumo \cite{Nagumo} and de Finetti \cite{deFinetti}, and in the two-variable case by Acz\'el \cite{Aczel66} through the bisymmetry equation. 
For characterizations without continuity assumptions, see Burai--Kiss--Szokol \cite{BuraiKissSzokol}.
It is well known that two generators define the same mean if and only if they differ by an affine transformation; see Hardy--Littlewood--P\'olya \cite[Theorem 83 and Section 3.7]{HLP}.

The Matkowski--Sut\^o problem asks for all pairs of quasi-arithmetic means whose sum is twice the arithmetic mean, that is, for all continuous strictly monotone functions $\varphi, \psi$ on $I$ such that
\begin{equation}\label{eq:MS1d}
A_{\varphi}(x,y) + A_{\psi}(x,y) = x + y, \qquad x, y \in I .
\end{equation}
The equation expresses that the arithmetic mean is invariant under the mean-type mapping $(A_{\varphi}, A_{\psi})$; by the invariance principle of Matkowski \cite{Matkowski99b}, this holds if and only if the Gauss composition of $A_{\varphi}$ and $A_{\psi}$, 
that is, the common limit of the iteration $x_{n+1} = A_{\varphi}(x_n, y_n)$, $y_{n+1} = A_{\psi}(x_n, y_n)$, $x_0 = x, y_0 = y$, is the arithmetic mean $\tfrac{1}{2} (x+y)$.
The prototype of such an invariance is the arithmetic--geometric mean of Gauss, which is invariant under the pair formed by the arithmetic and the geometric mean \cite{Cox}. 
For the invariance problem in general, we refer to the survey \cite{JarczykJarczyk} and to \cite{Nielsen23}.
Equation \eqref{eq:MS1d} was solved by Sut\^o \cite{Suto1, Suto2} in 1914 for analytic generators of a real variable, by Matkowski \cite{Matkowski} under twice continuous differentiability, and by Dar\'oczy and P\'ales \cite{DP} under continuity and strict monotonicity alone.
Up to affine changes of the generators, the solutions are the pair of identity maps $(t,t)$ and the exponential pairs $(\exp(pt), \exp(-pt))$ with $p \in \mathbb{R} \setminus \{0\}$.
Sut\^o's argument, which is discussed in Remark \ref{rem:suto} below, differentiates the equation with respect to each variable separately and reduces it, through the more general equation $f_1(\varphi_1(x) + \psi_1(y)) + f_2(\varphi_2(x) + \psi_2(y)) = x + y$, to a functional equation with separated variables solved in \cite{Suto1}; he also observed that the invariance of a quasi-arithmetic mean under a pair of quasi-arithmetic means reduces to \eqref{eq:MS1d} by a change of variable \cite[Section  3]{Suto2}, so that \eqref{eq:MS1d} governs the Gauss composition within the class of quasi-arithmetic means \cite{DP}.
The main contribution in \cite{DP} is a highly sophisticated regularity theory which derives the differentiability of the generators from their continuity. 

In this paper we consider two-dimensional versions of the problem, with generators defined on subsets of the plane, and we ask whether the rigidity of the one-dimensional problem persists.
The definition of a quasi-arithmetic mean uses only the arithmetic mean of the values of the generator and an inverse of the generator, so it makes sense for injective mappings from a subset of $\mathbb{R}^2$ into $\mathbb{R}^2$, provided that the mean of two values lies in the domain of the inverse.

We consider two frameworks for defining quasi-arithmetic means on subsets of the plane. 
In the first one, we identify $\mathbb{R}^2$ with $\mathbb{C}$ and take as generators the holomorphic injective functions on a convex domain $U \subset \mathbb{C}$ whose image is convex. 
The inverse is holomorphic on the image, and the resulting means are holomorphic in both variables.
In the second one, introduced recently by T\'oth \cite{Toth}, the generators are strictly increasing mappings in the sense of monotone operators, that is, $\langle f(x) - f(y), x - y \rangle > 0$ for $x \neq y$, defined on a {\it closed} convex set, and the inverse is replaced by the extended monotone left inverse on the convex hull of the image. 
Vector-valued quasi-arithmetic means also appear in \cite{Leonetti, NielsenGSI}.
T\'oth proved that on the whole space the arithmetic mean is generated only by affine maps with positive definite linear part \cite[Theorem 4]{Toth}, and posed the equality problem for these means. 
The two frameworks are not comparable. 
A holomorphic generator with convex image need not be monotone, and a monotone holomorphic map need not have a convex image; see Subsection~\ref{subsec:comparison}.

Our main results are Theorems \ref{thm:main} and \ref{thm:shear}, stated in Subsections \ref{subsec:holo} and \ref{subsec:mon}, respectively.
Theorem \ref{thm:main} solves the equation completely in the holomorphic framework: the solutions are exactly the identity pairs $(z,z)$ and the exponential pairs $(\exp(pz), \exp(-pz))$ with $p \in \mathbb{C} \setminus \{0\}$, up to complex affine changes of the generators.
Thus the rigidity of the one-dimensional problem persists, and two features are new. 
The exponent is now complex, and the existence of nontrivial solutions on a given domain $U$ depends on whether $\exp(pz)$ and $\exp(-pz)$ are injective on $U$ and have convex images. 
In Section \ref{sec:P}, we examine these conditions using elementary properties of the exponential function and the standard transformation formula for boundary curvature under holomorphic maps. 
The admissible exponents form a closed disc when $U$ is a disc, a real interval when $U$ is a horizontal strip, and reduce to $p = 0$ when $U$ is a bounded convex polygon or a half-plane.
No regularity theory is required, since holomorphy provides it from the beginning, and the argument differs from that of Sut\^o (Remark \ref{rem:suto}).

Theorem \ref{thm:shear} shows that the situation is entirely different in the monotone framework.
For every function $F$ on $\mathbb{R}$ with $|F(s) - F(t)| < 2|s - t|$ for $s \neq t$, the {\it shear mappings} $(x,y) \mapsto (x, y \pm F(x))$ are strictly increasing on $\mathbb{R}^2$ and form a solution of the two-variable equation with equal weights.
In Section \ref{sec:monotone}, we prove that these pairs also solve the weighted equation in any number of variables (Theorem \ref{thm:shear-weighted}).
Unless $F$ is affine, neither of the two means is the arithmetic mean, so the equation has an infinite-dimensional family of non-affine solutions on the whole plane, in contrast with the one-dimensional case.
The same construction extends to every real Hilbert space of dimension at least two, by acting on a two-dimensional subspace and as the identity on its orthogonal complement; see Remark \ref{rem:shear-features}(4).
Moreover, $F$ need not be differentiable, so no regularity theory in the spirit of \cite{DP} can hold for monotone generators in two dimensions.
For nonconstant $F$, these shear generators are not holomorphic, in accordance with Theorem \ref{thm:main}. 
On the other hand, the exponential pairs of Theorem \ref{thm:main} are strictly increasing after a suitable normalization (Subsection~\ref{subsec:comparison}).
The rigidity of the Matkowski--Sut\^o equation is a consequence of holomorphy rather than of monotonicity.

As a further property of these solution pairs, Section~\ref{sec:gauss} studies their simultaneous iteration.
For every initial pair, the iterates converge to the arithmetic mean in both the holomorphic and shear cases, so their Gauss composition is arithmetic.
For the exponential pairs with nonzero exponent, a change of variables reduces the iteration to repeated squaring; the shear pairs reach their common value after at most two iterations.

The paper is organized as follows.
Subsections \ref{subsec:holo} and \ref{subsec:mon} below introduce the two frameworks and state the main results. 
Section \ref{sec:proof} contains the fourth-order expansion at the diagonal and the proof of Theorem \ref{thm:main}. 
Section \ref{sec:P} describes the admissible exponents.
Section \ref{sec:monotone} introduces weighted means, proves the full version of Theorem \ref{thm:shear}, and discusses the properties of the shear solutions.
Section~\ref{sec:gauss} studies the Gauss composition of the holomorphic and shear solution pairs.
Section~\ref{sec:conclusion} compares the two frameworks and discusses open problems concerning the classification of solutions and the role of the domain.

\subsection{Holomorphic generators}\label{subsec:holo}

Throughout, a domain is a nonempty connected open subset of $\mathbb{C}$.
Let $U \subset \mathbb{C}$ be a convex domain.

\begin{Def}\label{def:H}
$\mathcal{H}(U)$ denotes the set of holomorphic injective functions $\varphi \colon U \to \mathbb{C}$ such that $\varphi(U)$ is convex.
For $\varphi \in \mathcal{H}(U)$, the quasi-arithmetic mean with generator $\varphi$ is
\begin{equation}\label{eq:qam}
A_{\varphi}(z,w) \coloneqq  \varphi^{-1}\left(\frac{\varphi(z)+\varphi(w)}{2}\right), \qquad z, w \in U .
\end{equation}
\end{Def}

By \cite[Theorem 10.33]{Rudin1987}, for $\varphi \in \mathcal{H}(U)$, the derivative $\varphi'$ has no zeros and the inverse map $\varphi^{-1} \colon \varphi(U) \to U$ is holomorphic.
Since $\varphi(U)$ is convex, the midpoint $\frac{\varphi(z)+\varphi(w)}{2}$ belongs to $\varphi(U)$, so \eqref{eq:qam} is well defined.
The map $A_{\varphi} \colon U \times U \to U$ is jointly holomorphic, symmetric, and satisfies $A_{\varphi}(z,z) = z$.

For $\varphi_1, \varphi_2 \in \mathcal{H}(U)$, we consider the Matkowski--Sut\^o equation
\begin{equation}\label{eq:MS}
A_{\varphi_1}(z,w) + A_{\varphi_2}(z,w) = z + w, \qquad z, w \in U .
\end{equation}
For $p \in \mathbb{C}$, let 
\begin{equation}\label{eq:chi}
\chi_p (z) \coloneqq \begin{cases} \exp(pz)  & p \ne 0\\ z & p = 0\end{cases}.
\end{equation}

Let 
\[ P_U \coloneqq  \left\{ p \in \mathbb{C} : \chi_p \in \mathcal{H}(U) \text{ and } \chi_{-p} \in \mathcal{H}(U) \right\}.\]

\begin{Thm}\label{thm:main}
Let $U \subset \mathbb{C}$ be a convex domain and let $\varphi_1, \varphi_2 \in \mathcal{H}(U)$.
Then the following are equivalent. \\
(1) The equation \eqref{eq:MS} holds. \\
(2) There exist $p \in \mathbb{C}$, $a_1, a_2 \in \mathbb{C} \setminus \{0\}$ and $b_1, b_2 \in \mathbb{C}$ such that $\varphi_1 = a_1 \chi_p + b_1$ and $\varphi_2 = a_2 \chi_{-p} + b_2$ on $U$. \\
In this case $p = \varphi_1''/\varphi_1'$ on $U$ and $p \in P_U$.
\end{Thm}

Two generators $\varphi, \psi \in \mathcal{H}(U)$ are called {\it affinely equivalent} 
if $\psi = a\varphi + b$ for some $a \in \mathbb{C} \setminus \{0\}$ and $b \in \mathbb{C}$; by Lemma \ref{lem:affine} below this holds if and only if $A_{\varphi} = A_{\psi}$.
Pairs are called affinely equivalent if they are so componentwise.

\begin{Cor}\label{cor:bijection}
Let $U \subset \mathbb{C}$ be a convex domain.
The map $p \mapsto (\chi_p, \chi_{-p})$ induces a bijection from $P_U$ onto the set of solutions $(\varphi_1,\varphi_2) \in \mathcal{H}(U)^2$ of \eqref{eq:MS} modulo affine equivalence.
\end{Cor}

If $p \in P_U \setminus \mathbb{R}$ and $U\cap\mathbb{R}\neq\varnothing$, 
then $A_{\chi_p}$ is not real-valued on $(U\cap\mathbb{R})^2$, so these complex exponents produce means which have no counterpart in the real theory.
In Section \ref{sec:P}, we show that $P_U$ is the closed disc of radius $1/r$ when $U$ is a disc of radius $r$ (Corollary \ref{cor:disc}), the real interval $[-\pi/d, \pi/d]$ when $U$ is a horizontal strip of width $d$ (Proposition \ref{prop:strip}), and $\{0\}$ when $U$ is a bounded convex polygon (Proposition \ref{prop:segment}) or when $U - U = \mathbb{C}$ (Proposition \ref{prop:basic}), for instance a half-plane or a sector.

\subsection{Strictly increasing generators}\label{subsec:mon}

In this subsection and in Section \ref{sec:monotone}, $\langle \cdot, \cdot \rangle$ and $|\cdot|$ denote the Euclidean inner product and norm on $\mathbb{R}^2$. 
Moreover, $ \operatorname{conv} (A)$ denotes the convex hull of $A \subset \mathbb{R}^2$.

\begin{Def}\label{def:increasing}
Let $C \subset \mathbb{R}^2$ be a nonempty closed convex set.
A mapping $f \colon C \to \mathbb{R}^2$ is {\it increasing} if $\langle f(x) - f(y), x - y \rangle \ge  0$ for all $x, y \in C$, and {\it strictly increasing} if $\langle f(x) - f(y), x - y \rangle > 0$ for all $x, y \in C$ with $x \neq y$.
\end{Def}

A strictly increasing mapping is injective.
By \cite[Theorems 2 and 3]{Toth}, for every strictly increasing $f \colon C \to \mathbb{R}^2$, there exists a unique increasing mapping $f^{(-1)} \colon \operatorname{conv} (f(C)) \to C$ such that $f^{(-1)}(f(x)) = x$ for every $x \in C$, and this mapping is continuous.
It is called the {\it extended monotone left inverse} of $f$. 
If $f(C)$ is convex, it is the ordinary inverse of $f$.

\begin{Def}\label{def:vqam}
Let $C \subset \mathbb{R}^2$ be a nonempty closed convex set and let $f \colon C \to \mathbb{R}^2$ be strictly increasing.
The two-variable quasi-arithmetic mean with generator $f$ is
\begin{equation*}
\mathcal{M}_f(x,y) \coloneqq f^{(-1)}\left(\frac{f(x)+f(y)}{2}\right), \qquad x,y \in C .
\end{equation*}
\end{Def}

The Matkowski--Sut\^o equation in this framework asks for strictly increasing mappings $f, g \colon C \to \mathbb{R}^2$ such that
\begin{equation}\label{eq:MSmon}
\mathcal{M}_f(x,y) + \mathcal{M}_g(x,y) = x+y, \qquad x,y \in C .
\end{equation}
This is the counterpart of \eqref{eq:MS} in the monotone framework.
Affine pairs $f(t) = A_1 t + b_1$, $g(t) = A_2 t + b_2$ with positive definite matrices $A_1, A_2$ are solutions, since both means are then arithmetic. 
Here a (not necessarily symmetric) real matrix $A$ is called positive definite when $\langle Ax,x\rangle > 0$ for every $x \neq 0$, equivalently when its symmetric part $(A+A^{\mathsf T})/2$ is positive definite. 
For a solution of \eqref{eq:MSmon}, if either mean is the arithmetic mean, then so is the other. 
On $C=\mathbb{R}^2$,  \cite[Theorem 4]{Toth} implies that both generators are affine. 
In dimension one, on $C = \mathbb{R}$, the same equation has the non-affine solutions $(e^{pt}, -e^{-pt})$, $p > 0$, by \cite{DP}, and both generators are strictly increasing.

For a function $F \colon \mathbb{R} \to \mathbb{R}$, define
\begin{equation}\label{eq:shear}
\theta_F(x,y) \coloneqq  ( x, \, y + F(x) ), \quad \eta_F(x,y) \coloneqq  ( x, \, y - F(x) ), \quad (x,y) \in \mathbb{R}^2 .
\end{equation}
Both maps are bijections of $\mathbb{R}^2$, and $\eta_F = \theta_{-F} = \theta_{F}^{-1}$.

\begin{Thm}\label{thm:shear}
Let $F \colon \mathbb{R} \to \mathbb{R}$ be a function. \\
(1) $\theta_F$ is strictly increasing if and only if $|F(s) - F(t)| < 2|s - t|$ for every $s, t \in \mathbb{R}$ with $s \neq t$. \\
(2) If the condition in (1) holds, then $\theta_F^{(-1)} = \eta_F$, $\eta_F^{(-1)} = \theta_F$, and
\begin{equation*}
\mathcal{M}_{\theta_F}(x,y) + \mathcal{M}_{\eta_F}(x,y) = x+y, \qquad x,y \in \mathbb{R}^2 .
\end{equation*}
(3) If the condition in (1) holds, then $\mathcal{M}_{\theta_F}(x,y) = \tfrac{1}{2}(x+y)$ for every $x, y \in \mathbb{R}^2$ if and only if $F$ is affine.
\end{Thm}

The condition in Theorem \ref{thm:shear}(1) holds for every Lipschitz function with Lipschitz constant less than $2$ and is sharp, and for every non-affine $F$ satisfying it, for instance $F(s) = \sin s$ or $F(s) = |s|$, Theorem \ref{thm:shear} provides a solution of \eqref{eq:MSmon} on $\mathbb{R}^2$ in which neither of the two means is the arithmetic mean.
The family of these solutions is infinite-dimensional, its members need not be differentiable, and they are not holomorphic unless $F$ is constant; see Section \ref{sec:monotone}.
Theorem \ref{thm:shear} is therefore the counterpart of Theorem \ref{thm:main} in the monotone framework. 
We emphasize that holomorphy forces the rigidity of the one-dimensional problem, whereas monotonicity does not.

\section{Expansion at the diagonal and proof of Theorem \ref{thm:main}}\label{sec:proof}

\begin{Lem}\label{lem:affine}
Let $U \subset \mathbb{C}$ be a convex domain and let $\varphi, \psi \in \mathcal{H}(U)$.
Then $A_{\varphi} = A_{\psi}$ on $U \times U$ if and only if there exist $a \in \mathbb{C}\setminus\{0\}$ and $b \in \mathbb{C}$ such that $\psi = a\varphi + b$ on $U$.
\end{Lem}

\begin{proof}
If $\psi = a\varphi + b$, then $\psi^{-1}(v) = \varphi^{-1}((v-b)/a)$ and $A_{\psi} = A_{\varphi}$ follows from \eqref{eq:qam}.

Conversely, assume $A_{\varphi} = A_{\psi}$.
The set $V \coloneqq  \varphi(U)$ is a convex domain and $\chi \coloneqq  \psi \circ \varphi^{-1}$ is holomorphic and injective on $V$.
Applying $\psi$ to $A_{\varphi}(z,w) = A_{\psi}(z,w)$ and writing $u = \varphi(z)$, $v = \varphi(w)$, we obtain
\begin{equation*}
\chi\left(\frac{u+v}{2}\right) = \frac{\chi(u)+\chi(v)}{2}, \qquad u, v \in V .
\end{equation*}
Differentiating with respect to $u$ gives $\chi'\left(\frac{u+v}{2}\right) = \chi'(u)$ for all $u, v \in V$.
For fixed $u \in V$ the set $\{ (u+v)/2 : v \in V \}$ is an open neighborhood of $u$.
Hence $\chi'$ is locally constant on $V$, and since $V$ is connected, $\chi' \equiv a$ for some constant $a$ by the identity theorem.
Thus $\chi(u) = au + b$ on $V$, and $a \neq 0$ because $\chi$ is injective.
Therefore $\psi = a\varphi + b$ on $U$.
\end{proof}

Now we consider the Taylor expansion of a quasi-arithmetic mean. 

For $\varphi \in \mathcal{H}(U)$ we write
\begin{equation}\label{eq:a}
a \coloneqq  \frac{\varphi''}{\varphi'} = (\log \varphi')' ,
\end{equation}
which is a holomorphic function on $U$.
It is invariant under affine changes of $\varphi$.

\begin{Lem}\label{lem:expansion}
Let $U \subset \mathbb{C}$ be a convex domain, $\varphi \in \mathcal{H}(U)$, and let $a$ be as in \eqref{eq:a}.
Then for every $z \in U$, as $h \to 0$,
\begin{equation}\label{eq:expansion}
A_{\varphi}(z, z+h) = z + \frac{h}{2} + \frac{a}{8} h^2 + \frac{a'}{16} h^3 + \frac{7a'' + 3aa' - 2a^3}{384} h^4 + O(h^5),
\end{equation}
where $a, a', a''$ are evaluated at $z$.
\end{Lem}

\begin{proof}
Fix $z \in U$.
From \eqref{eq:a},
\begin{equation*}
\varphi'' = a \varphi', \qquad \varphi''' = (a' + a^2) \varphi', \qquad \varphi'''' = (a'' + 3aa' + a^3) \varphi' .
\end{equation*}
Let $S(t) \coloneqq  \frac{\varphi(z+t) - \varphi(z)}{\varphi'(z)}$. 
Taylor expansion at $t = 0$ gives
\begin{equation*}
S(t) = t + \frac{a}{2} t^2 + \frac{a' + a^2}{6} t^3 + \frac{a'' + 3aa' + a^3}{24} t^4 + O(t^5) .
\end{equation*}
We see that $S(0) = 0$ and $S^{\prime}(0) = 1$. 
Comparing coefficients in $S(S^{-1}(\eta)) = \eta$, we obtain that 
\begin{equation*}
S^{-1}(\eta) = \eta - \frac{a}{2}\eta^2 + \frac{2a^2 - a'}{6} \eta^3 - \frac{6a^3 - 7aa' + a''}{24} \eta^4 + O(\eta^5) .
\end{equation*}
By \eqref{eq:qam}, 
\[ A_{\varphi}(z, z+h) = \varphi^{-1}\left( \varphi(z) + \varphi'(z) \eta \right) = z + S^{-1}(\eta)\] 
with
\begin{equation*}
\eta \coloneqq  \frac{S(h)}{2} = \frac{h}{2} + \frac{a}{4} h^2 + \frac{a'+a^2}{12} h^3 + \frac{a'' + 3aa' + a^3}{48} h^4 + O(h^5) .
\end{equation*}
Therefore, we have the following Taylor expansions: 
\begin{align*}
\eta^2 &= \frac{h^2}{4} + \frac{a}{4} h^3 + \left( \frac{a'+a^2}{12} + \frac{a^2}{16} \right) h^4 + O(h^5), \\
\eta^3 &= \frac{h^3}{8} + \frac{3a}{16} h^4 + O(h^5), \qquad
\eta^4 = \frac{h^4}{16} + O(h^5) .
\end{align*}
Collecting terms, the coefficient of $h^2$ in $S^{-1}(\eta)$ is $\frac{a}{4} - \frac{a}{8} = \frac{a}{8}$,
the coefficient of $h^3$ is $\frac{a'+a^2}{12} - \frac{a^2}{8} + \frac{2a^2 - a'}{48} = \frac{a'}{16}$,
and the coefficient of $h^4$ is
\begin{multline*}
\frac{a'' + 3aa' + a^3}{48} - \frac{a}{2}\left( \frac{a'+a^2}{12} + \frac{a^2}{16} \right) + \frac{2a^2 - a'}{6} \cdot \frac{3a}{16} - \frac{6a^3 - 7aa' + a''}{384} \\
= \frac{7a'' + 3aa' - 2a^3}{384} .
\end{multline*}
This proves \eqref{eq:expansion}.
\end{proof}

For $\varphi = \chi_p$ we have $a \equiv p$, and \eqref{eq:expansion} reads $A_{\chi_p}(z,z+h) = z + \frac{h}{2} + \frac{p}{8}h^2 - \frac{p^3}{192} h^4 + O(h^5)$.
For $p\ne0$ and each fixed $z\in U$, the following identity holds for all sufficiently small $h$:
\[ A_{\chi_p}(z,z+h) = z + \frac{h}{2} + \frac{1}{p} \operatorname{Log} \left(\cosh \frac{ph}{2} \right),\]
where $\operatorname{Log}$ denotes the principal branch of the logarithm.

\begin{Prop}\label{prop:necessary}
Let $U \subset \mathbb{C}$ be a convex domain and let $\varphi_1, \varphi_2 \in \mathcal{H}(U)$ satisfy \eqref{eq:MS}.
Put $\alpha_j \coloneqq  \varphi_j''/\varphi_j'$ for $j = 1,2$.
Then \\
(1) $\alpha_1 + \alpha_2 = 0$ on $U$, equivalently, $\varphi_1' \varphi_2'$ is constant on $U$. \\
(2) $\alpha_1$ is constant on $U$.
\end{Prop}

\begin{proof}
Fix $z \in U$.
For $h$ in a neighborhood of $0$, 
$A_{\varphi_1}(z,z+h) + A_{\varphi_2}(z,z+h) - (2z+h)$ 
is a holomorphic function of $h$ which vanishes identically.
By Lemma \ref{lem:expansion} its Taylor coefficients at $h = 0$ are
\begin{equation*}
\frac{\alpha_1 + \alpha_2}{8}, \qquad \frac{\alpha_1' + \alpha_2'}{16}, \qquad \frac{7(\alpha_1'' + \alpha_2'') + 3(\alpha_1 \alpha_1' + \alpha_2 \alpha_2') - 2(\alpha_1^3 + \alpha_2^3)}{384}
\end{equation*}
for $h^2, h^3, h^4$, respectively, and all of them vanish.

The vanishing of the $h^2$-coefficient gives $\alpha_1 + \alpha_2 = 0$ on $U$.
Since $(\varphi_1'\varphi_2')' = (\alpha_1 + \alpha_2)\varphi_1'\varphi_2' = 0$, this is equivalent to $\varphi_1'\varphi_2'$ being constant.
This proves (1).

Substituting $\alpha_2 = -\alpha_1$ into the $h^4$-coefficient, the terms $\alpha_1'' + \alpha_2''$ and $\alpha_1^3 + \alpha_2^3$ vanish, and the coefficient reduces to $\frac{6 \alpha_1 \alpha_1'}{384}$.
Hence $\alpha_1 \alpha_1' = 0$, that is, $(\alpha_1^2)' = 0$ on $U$.
Since $U$ is connected, $\alpha_1^2 \equiv c$ for a constant $c$.
If $c = 0$, then $\alpha_1 \equiv 0$.
If $c \neq 0$, then $\alpha_1$ is a continuous function on the connected set $U$ with values in the two-point set $\{\sqrt{c}, -\sqrt{c}\}$, hence constant.
This proves (2).
\end{proof}

\begin{Rem}\label{rem:local}
The proof of Proposition \ref{prop:necessary} uses \eqref{eq:MS} only for $(z,w)$ in a neighborhood of the diagonal $\{(z,z) : z \in U\}$.
The convexity of $U$ and of $\varphi_j(U)$ serves only to define $A_{\varphi_j}$ on all of $U \times U$.
\end{Rem}

We show that exponential pairs are solutions. 

\begin{Prop}\label{prop:sufficient}
Let $U \subset \mathbb{C}$ be a convex domain and let $p \in P_U$.
Then $(\chi_p, \chi_{-p})$ satisfies \eqref{eq:MS}.
\end{Prop}

\begin{proof}
For $p = 0$ both means are the arithmetic mean and \eqref{eq:MS} is clear.
Let $p \ne 0$, $f \coloneqq  \chi_p$, $g \coloneqq  \chi_{-p}$, so that $g(\zeta) f(\zeta) = 1$ for all $\zeta \in \mathbb{C}$.
For $(z,w) \in U \times U$ put $m \coloneqq  A_f(z,w) \in U$, so that $f(m) = \frac{f(z)+f(w)}{2}$.
Then
\begin{align*}
g(z+w-m) &= \exp(-p(z+w)) \exp(pm) \\
&= \exp(-p(z+w)) \cdot \frac{\exp(pz) + \exp(pw)}{2} \\
&= \frac{\exp(-pw) + \exp(-pz)}{2} = \frac{g(z)+g(w)}{2} = g\left( A_g (z,w) \right) .
\end{align*}
Let $\Omega \coloneqq  \left\{ (z,w) \in U \times U : z + w - A_f(z,w) \in U \right\}$.
Since $A_f$ is continuous and $U$ is open, $\Omega$ is open in $\mathbb{C}^2$ and it contains the diagonal because $A_f(z,z) = z$.
For $(z,w) \in \Omega$,  $z+w-m \in U$ and $A_g(z,w) \in U$. 
They have the same image under $g$. 
Since $g$ is injective on $U$, $z + w - A_f(z,w) = A_g(z,w)$ on $\Omega$.
The function $\Phi(z,w) \coloneqq  A_f(z,w) + A_g(z,w) - z - w$ is holomorphic on $U \times U$, which is connected, and vanishes on the nonempty open set $\Omega$. 
For each fixed $z\in U$, the holomorphic function $w \mapsto \Phi(z,w)$ vanishes in an open neighborhood of $z$. 
By the one-variable identity theorem, it vanishes throughout $U$. 
Since $z$ was arbitrary, $\Phi = 0$ on $U\times U$.
\end{proof}

\begin{proof}[Proof of Theorem \ref{thm:main}]
(1) $\Rightarrow$ (2).
By Proposition \ref{prop:necessary}(2), $\varphi_1''/\varphi_1' \equiv p$ for some $p \in \mathbb{C}$.
Then $\left(\varphi_1' \exp(-pz)\right)' = (\varphi_1'' - p\varphi_1') \exp(-pz) = 0$, so $\varphi_1' = C \exp(pz)$ with a constant $C$.
Since $\varphi_1'$ does not vanish, $C \ne 0$. 
If $p \ne 0$, then $\varphi_1 = \frac{C}{p} \exp(pz) + b_1$ for some $b_1 \in \mathbb{C}$, and if $p = 0$, then $\varphi_1 = Cz + b_1$ for some $b_1 \in \mathbb{C}$.
In both cases $\varphi_1 = a_1 \chi_p + b_1$ with $a_1 \ne 0$.
By Proposition \ref{prop:necessary}(1), 
$\varphi_2' = c/\varphi_1' = \frac{c}{C} \exp(-pz)$ with a constant $c \ne 0$.  
The same argument gives that $\varphi_2 = a_2\chi_{-p}+b_2$ on $U$ for some $a_2\in \mathbb{C} \setminus\{0\}$ and $b_2 \in \mathbb C$.

(2) $\Rightarrow$ (1).
Affine maps preserve injectivity and convexity of the image, so $\chi_p, \chi_{-p} \in \mathcal{H}(U)$, that is, $p \in P_U$.
By Lemma \ref{lem:affine}, $A_{\varphi_1} = A_{\chi_p}$ and $A_{\varphi_2} = A_{\chi_{-p}}$, and \eqref{eq:MS} follows from Proposition \ref{prop:sufficient}.

The last assertions follow from $\varphi_1''/\varphi_1' = p$ and $p\in P_U$ because $\chi_p, \chi_{-p}\in \mathcal{H}(U)$.
\end{proof}

\begin{proof}[Proof of Corollary \ref{cor:bijection}]
For $p \in P_U$ the pair $(\chi_p, \chi_{-p})$ is a solution by Proposition \ref{prop:sufficient}, and every solution is affinely equivalent to such a pair by Theorem \ref{thm:main}.
If $(\chi_p, \chi_{-p})$ and $(\chi_q, \chi_{-q})$ are affinely equivalent, then $\chi_q = a\chi_p + b$, so $\chi_q' = a\chi_p'$ and $q = \chi_q''/\chi_q' = \chi_p''/\chi_p' = p$.
\end{proof}

\begin{Rem}\label{rem:suto}
Proposition \ref{prop:necessary} and the resulting uniqueness part of Theorem \ref{thm:main} are complex counterparts of the result of Sut\^o \cite[Section  2]{Suto2}, and the two arguments are different.
Sut\^o works with analytic functions of independent real variables whose inverses are single-valued \cite[p.~82 and p.~99]{Suto2}.
He treats the equation $f_1(\varphi_1(x) + \psi_1(y)) + f_2(\varphi_2(x) + \psi_2(y)) = x + y$ in six unknown functions.
Differentiating with respect to $x$ and with respect to $y$, he eliminates $f_1'$ and $f_2'$, and in the generic case the substitution $\xi = \varphi_1(x)$, $\eta = \psi_1(y)$ leads to an equation of the form
\begin{equation*}
f_1'(\xi + \eta) = \frac{A(\xi) + B(\eta)}{C(\xi) + D(\eta)},
\end{equation*}
whose analytic solutions he had determined in \cite[Section  3]{Suto1} by differentiation, by the linear independence of derivatives, and by solving ordinary differential equations.
A case analysis then produces a list of solution families, and the Matkowski--Sut\^o equation is the special case $\psi_i = \varphi_i$, $f_i(X) = \varphi_i^{-1}(X/2)$, for which only the exponential pairs and the affine pairs remain \cite[Section  2]{Suto2}.
Thus Sut\^o classifies a more general equation in six unknown functions by a case analysis, whereas the proof of Proposition \ref{prop:necessary} uses a fourth-order expansion at the diagonal to determine the possible generators.
Only the germ of \eqref{eq:MS} along the diagonal enters this necessity argument (Remark \ref{rem:local}), and the sufficiency in Proposition \ref{prop:sufficient} follows from the identity theorem.
The complex exponent $p$ and the dependence of $P_U$ on the domain have no counterpart in Sut\^o's real setting.
Sut\^o also noted that the invariance of a quasi-arithmetic mean $A_{\psi}$ under a pair $(A_{\psi_1}, A_{\psi_2})$ reduces to \eqref{eq:MS} for the generators $\psi_1 \circ \psi^{-1}$ and $\psi_2 \circ \psi^{-1}$ by the change of variable $\xi = \psi(x)$, and he recorded the invariance of the geometric mean under the arithmetic--harmonic pair as the example $\psi (x) = \log x$ \cite[Section  3]{Suto2}.
\end{Rem}

\section{Admissible exponents}\label{sec:P}

By Corollary \ref{cor:bijection}, the solutions of \eqref{eq:MS} on $U$ are parametrized by $P_U$.
This section describes $P_U$.

We first give elementary properties. 

\begin{Prop}\label{prop:basic}
Let $U \subset \mathbb{C}$ be a convex domain. \\
(1) $0 \in P_U$ and $P_U = -P_U$. \\
(2) $P_{c + \lambda U} = \lambda^{-1} P_U$ for every $c \in \mathbb{C}$ and $\lambda \in \mathbb{C}\setminus\{0\}$. \\
(3) Let $p \in \mathbb{C}\setminus\{0\}$ and let $K \subset \mathbb{C}$ be a nonempty convex set. Then $\chi_p$ is injective on $K$ if and only if $2\pi i/p \notin K - K$. \\
(4) If $U - U = \mathbb{C}$, then $P_U = \{0\}$.
\end{Prop}

Part (3) is stated for convex sets because it will also be applied to the closure of a convex domain.

\begin{proof}
(1) is immediate from the definition of $P_U$. 

(2) For $z \in U$ we have $\chi_p(c + \lambda z) = \exp(pc) \chi_{p\lambda}(z)$ if $p \neq 0$, and $\chi_0(c + \lambda z) = c + \lambda \chi_0(z)$.
Since affine maps preserve injectivity and convexity of the image, $\chi_p \in \mathcal{H}(c + \lambda U)$ if and only if $\chi_{p\lambda} \in \mathcal{H}(U)$.
The same holds with $-p$ in place of $p$, and (2) follows.

(3) For $z, w \in K$ we have $\exp(pz) = \exp(pw)$ if and only if $z - w \in \frac{2\pi i}{p} \mathbb{Z}$.
Hence $\chi_p$ is injective on $K$ if and only if $(K - K) \cap \frac{2\pi i}{p}\mathbb{Z} = \{0\}$.
The set $K - K$ is convex, symmetric, and contains $0$.
If $\frac{2\pi i k}{p} \in K - K$ for some integer $k \neq 0$, we may assume $k \ge  1$ by symmetry, and then $\frac{2\pi i}{p} = \frac{1}{k} \cdot \frac{2\pi i k}{p} + \left( 1 - \frac{1}{k} \right) \cdot 0 \in K - K$ by convexity.
This proves (3).

(4) If $p \neq 0$, then $2\pi i/p \in \mathbb{C} = U - U$, so $\chi_p$ is not injective on $U$ by (3).
\end{proof}

Proposition \ref{prop:basic}(4) applies to half-planes and to sectors $\{ z \neq 0 : \alpha < \arg z < \beta \}$ with $0 < \beta - \alpha \le  \pi$, since the difference set of a convex cone with nonempty interior is a linear subspace with nonempty interior, hence equal to $\mathbb{C}$.

For a regular $C^2$ curve $\Gamma$ parametrized by $\gamma$, the signed curvature is
\begin{equation}\label{eq:curvature}
\kappa_{\Gamma} = \frac{\operatorname{Im}\left( \gamma'' \, \overline{\gamma'} \right)}{|\gamma'|^3} .
\end{equation}
If $\gamma$ is parametrized by arc length with unit tangent $\tau \coloneqq  \gamma'$, then $\gamma'' = i \kappa_{\Gamma} \tau$.

For a domain $D$ whose boundary is a $C^2$ Jordan curve, 
the positive orientation of $\partial D$ is the one for which $D$ lies on the left, and then $D$ is convex if and only if the signed curvature of $\partial D$ is nonnegative everywhere; see \cite[Problems 1.7.5 and 1.7.6, pp.~26--27]{Toponogov2006}. 

\begin{Lem}\label{lem:curvature}
Let $U \subset \mathbb{C}$ be a bounded convex domain whose boundary $\partial U$ is a $C^2$ Jordan curve, parametrized by arc length by $\gamma = \gamma(s)$ with positive orientation, let $\tau \coloneqq  \gamma'$, and let $\kappa$ denote the signed curvature of $\partial U$.
Let $f$ be holomorphic on an open set containing $\overline{U}$, injective on $\overline{U}$, and such that $f' \neq 0$ on $\partial U$.
Then $f(U)$ is a convex domain if and only if
\begin{equation}\label{eq:criterion}
\kappa(s) + \operatorname{Im}\left( \tau(s) \, \frac{f''(\gamma(s))}{f'(\gamma(s))} \right) \ge  0 \qquad \text{for all } s .
\end{equation}
\end{Lem}

\begin{proof}
Since $f$ is continuous and injective on the compact set $\overline{U}$, it is a homeomorphism of $\overline{U}$ onto $f(\overline{U})$.
The set $f(U)$ is open by the open mapping theorem, $f(\overline{U}) = \overline{f(U)}$, and therefore $\partial f(U) = f(\overline{U}) \setminus f(U) = f(\partial U)$.
Thus $\Gamma \coloneqq  f \circ \gamma$ parametrizes $\partial f(U)$. 

It is a $C^2$ Jordan curve, regular because $f' \ne 0$ on $\partial U$, and positively oriented because $f$ preserves orientation. 
Since $\gamma$ is parametrized by arc length, $|\tau| = |\gamma^{\prime}| = 1$. 
From $\Gamma' = f'(\gamma)\tau$ and $\Gamma'' = f''(\gamma)\tau^2 + i\kappa f'(\gamma)\tau$, 
we obtain that 
\begin{equation*}
\Gamma'' \, \overline{\Gamma'} = f''(\gamma)\overline{f'(\gamma)} \, \tau + i\kappa |f'(\gamma)|^2. 
\end{equation*}
By this and \eqref{eq:curvature}, we obtain that 
\begin{equation}\label{eq:image-curvature}
\kappa_{\Gamma} = \frac{\kappa + \operatorname{Im}\left( \tau f''(\gamma)/f'(\gamma) \right)}{|f'(\gamma)|} .
\end{equation}
The domain $f(U)$ is convex if and only if $\kappa_{\Gamma} \ge  0$ everywhere, which is \eqref{eq:criterion}.
\end{proof}

The next lemma gives a local form of the necessary condition in Lemma~\ref{lem:curvature}. 
It requires boundary regularity and holomorphic extension only near the boundary point under consideration, and allows $U$ to be unbounded.

\begin{Lem}\label{lem:local}
Let $U \subset \mathbb{C}$ be a convex domain and let $f$ be holomorphic and injective on $U$ with $f(U)$ convex.
Let $\zeta \in \partial U$.
Assume that $f$ extends holomorphically to a neighborhood of $\zeta$ with $f'(\zeta) \neq 0$, and that there is a disc $B$ centered at $\zeta$ such that $\partial U \cap B$ is a regular $C^2$ arc with $U \cap B$ on its left.
Let $\tau$ be the unit tangent and $\kappa$ the signed curvature of this arc at $\zeta$.
Then
\begin{equation}\label{eq:local}
\kappa + \operatorname{Im}\left( \tau \, \frac{f''(\zeta)}{f'(\zeta)} \right) \ge  0 .
\end{equation}
\end{Lem}

\begin{proof}
Shrinking $B$, we may assume that $f$ is a diffeomorphism of $B$ onto an open set; see Figure~\ref{fig:local-curvature}.
Then $\Gamma \coloneqq  f(\partial U \cap B)$ is a regular $C^2$ arc through $f(\zeta)$ with $f(U \cap B)$ on its left, and by the computation leading to \eqref{eq:image-curvature} its signed curvature at $f(\zeta)$ equals $\frac{ \kappa + \operatorname{Im}( \tau f''(\zeta)/f'(\zeta) ) }{ |f'(\zeta)| }$.
Therefore, it suffices to show that the signed curvature of $\Gamma$ at $f(\zeta)$ is nonnegative.

We first show that $f(\zeta) \notin f(U)$.
Otherwise $f(\zeta) = f(z_0)$ for some $z_0 \in U$, and $z_0 \neq \zeta$ because $\zeta \notin U$.
Choose a disc $B_0 \subset U$ centered at $z_0$ with $\zeta \notin \overline{B_0}$.
By the open mapping theorem, $f(B_0)$ is a neighborhood of $f(z_0) = f(\zeta)$.
For $z \in U$ close to $\zeta$ we have $z \notin \overline{B_0}$ and $f(z) \in f(B_0)$, so $f(z) = f(z')$ for some $z' \in B_0$ with $z' \ne z$, contradicting the injectivity of $f$ on $U$.

Since $f(U)$ is open and convex and $f(\zeta) \in \overline{f(U)} \setminus f(U)$, there is a line $L$ through $f(\zeta)$ and an open half-plane $H$ bounded by $L$ with $f(U) \subset H$.
Since the arc $\Gamma \subset \overline{f(U)} \subset \overline{H}$ passes through $f(\zeta) \in L$, it is tangent to $L$ at $f(\zeta)$.
For small $t > 0$ the point $\zeta + i t \tau$ lies in $U \cap B$, and $f(\zeta + it\tau) = f(\zeta) + i t \tau f'(\zeta) + O(t^2)$ lies in $H$. 
Since $i\tau f'(\zeta)$ is a positive multiple of the left unit normal of $\Gamma$ at $f(\zeta)$, the half-plane $H$ is the left half-plane of $L$ with respect to the orientation of $\Gamma$. 
Here the left unit normal is obtained by rotating the unit tangent counterclockwise by $\pi/2$. 
Thus $\Gamma$ is contained in the closed left half-plane of its tangent line at $f(\zeta)$. 
Choose orthonormal coordinates centered at $f(\zeta)$, with the positive horizontal axis in the tangent direction and the positive vertical axis in the left normal direction. 
Then $\Gamma$ is locally the graph $y=n(x)$ of a $C^2$ function $n$.
Writing $\Gamma$ near $f(\zeta)$ as a graph over $L$ in the direction of the left normal, the graph function $n$ satisfies $n \ge  0$, $n(0) = 0$ and $n'(0) = 0$, hence $n''(0) \ge  0$, and the signed curvature of $\Gamma$ at $f(\zeta)$ equals $n''(0) \ge  0$.
\end{proof}

\begin{figure}[htbp]
\centering
\includegraphics[width=\linewidth]{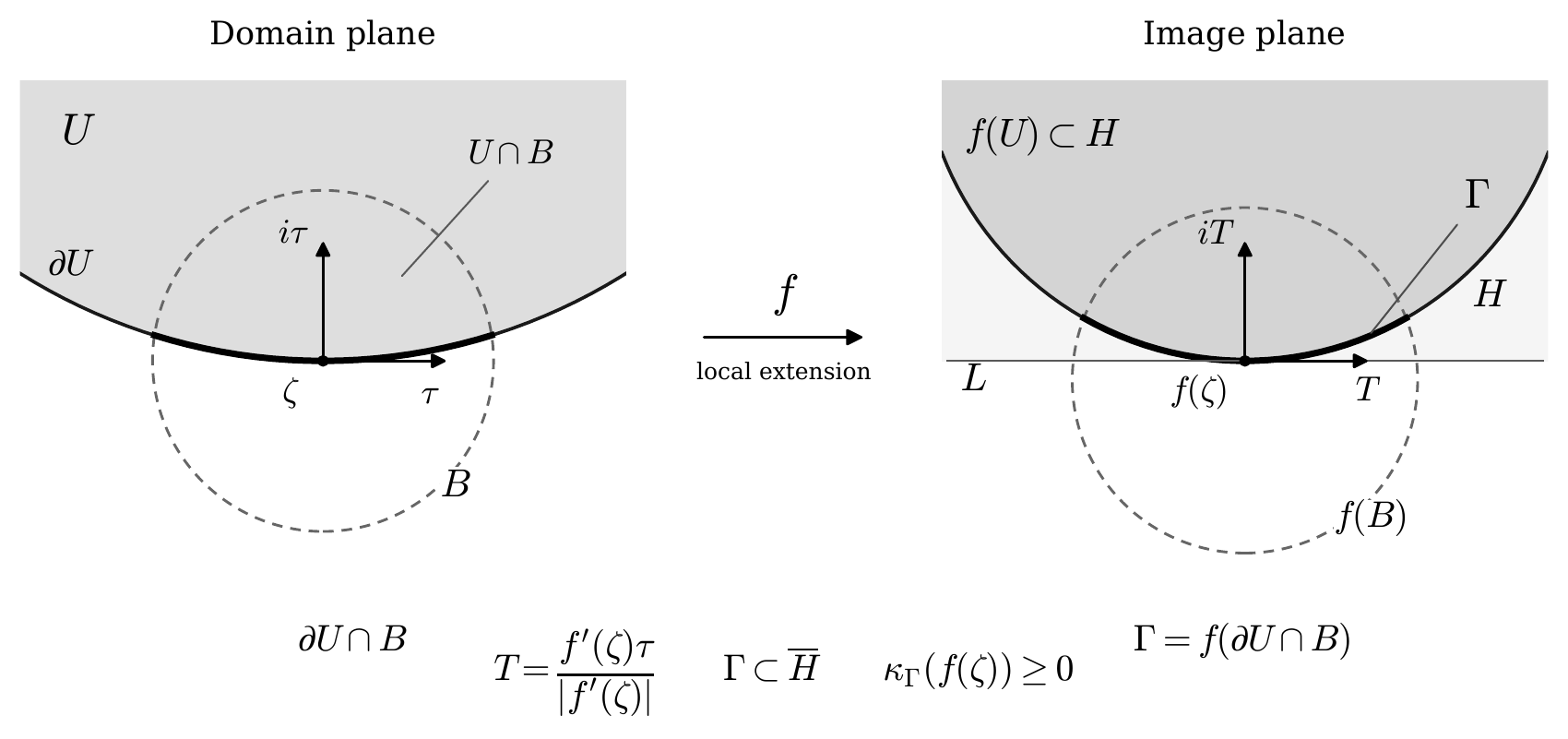}
\caption{Local geometry in Lemma~\ref{lem:local}.
The holomorphic extension of $f$ maps a neighborhood $B$ of $\zeta$
diffeomorphically onto $f(B)$.
The supporting line $L$ of $f(U)$ at $f(\zeta)$ is tangent to $\Gamma = f(\partial U\cap B)$, and $\Gamma$ lies in the closed left half-plane $\overline H$. 
The unit tangent is $T=f'(\zeta)\tau/|f'(\zeta)|$. 
}
\label{fig:local-curvature}
\end{figure}

For $f = \chi_p$ we have $f''/f' \equiv p$, and \eqref{eq:criterion} and \eqref{eq:local} become
\begin{equation}\label{eq:criterion-exp}
\kappa + \operatorname{Im}(p \tau) \ge  0 .
\end{equation}

\begin{Cor}[disks]\label{cor:disc}
Let $U \subset \mathbb{C}$ be a convex domain. \\
(1) If $U$ is bounded, $\partial U$ is a $C^2$ Jordan curve, and $\kappa \ge  \kappa_0 > 0$ on $\partial U$, then
$\{ p \in \mathbb{C} : |p| \le  \kappa_0, \ |p| \operatorname{diam}(U) < 2\pi \} \subset P_U$. \\
(2) If $U = \{ z \in \mathbb{C} : |z - c| < r \}$ with $c \in \mathbb{C}$ and $r > 0$, then $P_U = \{ p \in \mathbb{C} : |p| \le  1/r \}$.
\end{Cor}

\begin{proof}
(1) Let $|p| \le  \kappa_0$ and $|p| \operatorname{diam}(U) < 2\pi$.
Since $0 \in P_U$, we may assume that $p \ne 0$.
For $\zeta, \zeta' \in \overline{U}$ we have $|\zeta - \zeta'| \le  \operatorname{diam}(U) < 2\pi/|p| = |2\pi i/p|$, so $2\pi i/p \notin \overline{U} - \overline{U}$ and likewise $-2\pi i/p \notin \overline{U} - \overline{U}$.
By Proposition \ref{prop:basic}(3) with $K = \overline{U}$, both $\chi_p$ and $\chi_{-p}$ are injective on $\overline{U}$.
Moreover $\chi_{\pm p}' = \pm p \, \chi_{\pm p}$ has no zeros.
By \eqref{eq:criterion-exp}, $\kappa + \operatorname{Im}(\pm p\tau) \ge  \kappa_0 - |p| \ge  0$ on $\partial U$, so Lemma \ref{lem:curvature} shows that $\chi_p(U)$ and $\chi_{-p}(U)$ are convex.
Hence $p \in P_U$.

(2) By Proposition \ref{prop:basic}(2), we can assume that $c = 0$.
Since $\kappa \equiv 1/r$ and $\operatorname{diam}(U) = 2r$, part (1) gives $\{ |p| \le  1/r \} \subset P_U$. 

Let $p \in P_U \setminus \{0\}$.
Choose $\theta \in \mathbb{R}$ with $p \exp(i\theta) = -|p|$ and put $\zeta \coloneqq  r \exp(i\theta)$, so that the positively oriented unit tangent at $\zeta$ is $\tau = i \exp(i\theta)$.
Lemma \ref{lem:local} applied to $f = \chi_p$ at $\zeta$ gives, by \eqref{eq:criterion-exp},
\begin{equation*}
0 \le  \frac{1}{r} + \operatorname{Im}(p \, i \exp(i\theta)) = \frac{1}{r} + \operatorname{Re}(p \exp(i\theta)) = \frac{1}{r} - |p| ,
\end{equation*}
hence $|p| \le  1/r$.
\end{proof}

\begin{Prop}[flat boundary pieces]\label{prop:segment}
Let $U \subset \mathbb{C}$ be a convex domain. \\
(1) If $\partial U$ contains a nondegenerate line segment with direction $\tau \in \mathbb{C}$, $|\tau| = 1$, then $P_U \subset \{ p \in \mathbb{C} : p\tau \in \mathbb{R} \}$. \\
(2) If $\partial U$ contains two nonparallel nondegenerate line segments, then $P_U = \{0\}$. 
In particular, $P_U = \{0\}$ for every bounded convex polygon $U$.
\end{Prop}

\begin{proof}
(1) Let $p \in P_U$ and let $\zeta$ be an interior point of the segment.
Replacing $\tau$ by $-\tau$ if necessary, we may assume that $U$ lies on the left of the segment oriented by $\tau$; the assertion is invariant under this replacement.
Near $\zeta$ the boundary of $U$ is the segment, so the hypotheses of Lemma \ref{lem:local} hold with $\kappa = 0$.
Applying Lemma \ref{lem:local} to $f = \chi_p$ and to $f = \chi_{-p}$, both of which belong to $\mathcal{H}(U)$, we obtain from \eqref{eq:criterion-exp} that $\operatorname{Im}(p\tau) \ge  0$ and $\operatorname{Im}(-p\tau) \ge  0$.
Hence $p\tau \in \mathbb{R}$.

(2) Let $\tau_1, \tau_2$ be the directions of the two segments and let $p \in P_U$.
By (1), $p\tau_1, p\tau_2 \in \mathbb{R}$.
If $p \neq 0$, then $\tau_1/\tau_2 = (p\tau_1)/(p\tau_2) \in \mathbb{R}$, so the segments are parallel, a contradiction.
A bounded convex polygon has at least three sides and hence two nonparallel sides.
\end{proof}

\begin{Prop}[strips]\label{prop:strip}
For $d > 0$ let $S_d \coloneqq  \{ z \in \mathbb{C} : |\operatorname{Im} z| < d/2 \}$.
Then $P_{S_d} = \{ p \in \mathbb{R} : |p| \le  \pi/d \}$.
\end{Prop}

\begin{proof}
The boundary of $S_d$ contains segments with direction $1$, so $P_{S_d} \subset \mathbb{R}$ by Proposition \ref{prop:segment}(1).
Since $S_d = -S_d$, it suffices to determine for which $p \in \mathbb{R}\setminus\{0\}$ we have $\chi_p \in \mathcal{H}(S_d)$.
We have $S_d - S_d = S_{2d}$, and $2\pi i/p \in S_{2d}$ if and only if $2\pi/|p| < d$.
By Proposition \ref{prop:basic}(3), $\chi_p$ is injective on $S_d$ if and only if $|p| d \le  2\pi$.

If $|p| d \le  \pi$, then $\chi_p(S_d) = \left\{ \rho \exp(i\theta) : \rho > 0, \ |\theta| < |p| d/2 \right\}$ is an open sector of opening angle $|p| d \le  \pi$, which is convex; hence $\chi_p \in \mathcal{H}(S_d)$.

If $\pi < |p| d \le  2\pi$, then $\chi_p(S_d)$ does not contain $0$ but contains the points $e^{ipy}$ with $|y| < d/2$, whose arguments fill an open interval of length $|p|d > \pi$.
A convex open set not containing $0$ is contained in an open half-plane $\{ w : \operatorname{Re}(\exp(-i\alpha) w) > 0 \}$ for some $\alpha \in \mathbb{R}$, and the arguments of the points of such a half-plane fill an interval of length $\pi$.
Hence $\chi_p (S_d)$ is not convex.

Therefore $\chi_p \in \mathcal{H}(S_d)$ if and only if $|p| d \le  \pi$, and the assertion follows.
\end{proof}

\begin{figure}[htbp]
\centering
\includegraphics[width=0.9\linewidth]{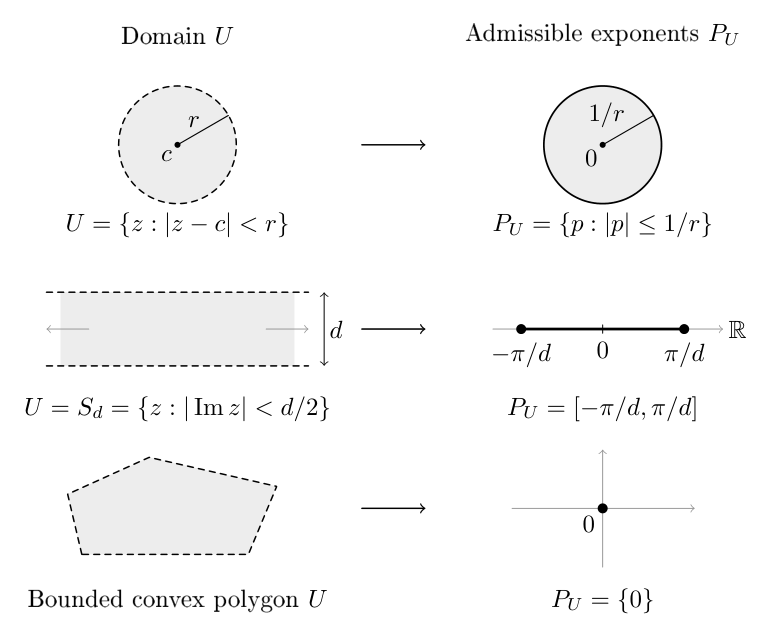}
\caption{Examples of $U\mapsto P_U$ (schematic). 
Half-planes also have $P_U=\{0\}$.}
\label{fig:domain-exponents}
\end{figure}

\begin{Rem}\label{rem:shape}
The examples show that the shape of $P_U$ reflects the geometry of $\partial U$.
Figure \ref{fig:domain-exponents} summarizes three representative cases.
A bounded domain with a $C^2$ boundary of strictly positive curvature admits a disc of exponents centered at $0$ (Corollary \ref{cor:disc}), a flat piece of boundary confines the exponent to a line through $0$ (Proposition \ref{prop:segment}(1)), two flat pieces in different directions or a difference set equal to $\mathbb{C}$ leave only the trivial solution (Propositions \ref{prop:segment}(2) and \ref{prop:basic}(4)), and the strip realizes the intermediate case of a real interval (Proposition \ref{prop:strip}).
\end{Rem}

\FloatBarrier
\section{Weighted means and properties of the shear solutions}\label{sec:monotone}

\subsection{Weighted means and shear pairs}

We extend Definition \ref{def:vqam} to weighted means in any number of variables and state the full version of Theorem \ref{thm:shear}. 
For notation, we follow \cite[Section  4]{Toth}. 

\begin{Def}\label{def:vqam-weighted}
Let $C \subset \mathbb{R}^2$ be a nonempty closed convex set, let $f \colon C \to \mathbb{R}^2$ be strictly increasing, let $m \ge  2$, and let $\lambda = (\lambda_1, \dots, \lambda_m)$ with $\lambda_i > 0$ and $\lambda_1 + \cdots + \lambda_m = 1$.
The $m$-variable weighted quasi-arithmetic mean with generator $f$ and the weighted arithmetic mean are
\begin{equation*}
\begin{aligned}
\mathcal{M}^{[m]}_{f,\lambda}(x_1, \dots, x_m)
&\coloneqq f^{(-1)}\left(\sum_{i=1}^m \lambda_i f(x_i)\right), \\
\mathcal{A}^{[m]}_{\lambda}(x_1, \dots, x_m)
&\coloneqq \sum_{i=1}^m \lambda_i x_i,
\end{aligned}
\end{equation*}
for $x_1, \dots, x_m \in C$.
For $m=2$ and $\lambda=(\tfrac{1}{2},\tfrac{1}{2})$, this agrees with Definition \ref{def:vqam}: $\mathcal{M}^{[2]}_{f,(1/2,1/2)}=\mathcal{M}_f$.
\end{Def}

For strictly increasing mappings $f,g \colon C \to \mathbb{R}^2$, the weighted Matkowski--Sut\^o equation is
\begin{equation}\label{eq:MSmon-weighted}
\mathcal{M}^{[m]}_{f,\lambda}(x_1, \dots, x_m)
+\mathcal{M}^{[m]}_{g,\lambda}(x_1, \dots, x_m)
=2\,\mathcal{A}^{[m]}_{\lambda}(x_1, \dots, x_m)
\end{equation}
for every $x_1, \dots, x_m \in C$.
Its two-variable case with equal weights is \eqref{eq:MSmon}.

\begin{Thm}\label{thm:shear-weighted}
Let $F \colon \mathbb{R} \to \mathbb{R}$ be a function, and let $\theta_F$ and $\eta_F$ be defined by \eqref{eq:shear}. \\
(1) $\theta_F$ is strictly increasing if and only if $|F(s)-F(t)|<2|s-t|$ for every $s,t\in\mathbb{R}$ with $s\neq t$. \\
(2) If the condition in (1) holds, then $\theta_F^{(-1)}=\eta_F$, $\eta_F^{(-1)}=\theta_F$, and for every $m\ge 2$ and every $\lambda=(\lambda_1,\dots,\lambda_m)$ with $\lambda_i>0$ and $\sum_{i=1}^m\lambda_i=1$,
\begin{equation}\label{eq:shearMS-weighted}
\mathcal{M}^{[m]}_{\theta_F,\lambda}(x_1,\dots,x_m)
+\mathcal{M}^{[m]}_{\eta_F,\lambda}(x_1,\dots,x_m)
=2\,\mathcal{A}^{[m]}_{\lambda}(x_1,\dots,x_m)
\end{equation}
for every $x_1,\dots,x_m\in\mathbb{R}^2$. \\
(3) If the condition in (1) holds, then $\mathcal{M}_{\theta_F}(x,y)=\tfrac{1}{2}(x+y)$ for every $x,y\in\mathbb{R}^2$ if and only if $F$ is affine.
\end{Thm}

\begin{proof}
(1) Let $P = (s_1, t_1)$ and $Q = (s_2, t_2)$, and put $a \coloneqq  s_1 - s_2$, $b \coloneqq  t_1 - t_2$, $c \coloneqq  F(s_1) - F(s_2)$.
Then
\begin{equation}\label{eq:abc}
\langle \theta_F(P) - \theta_F(Q), P - Q \rangle = a^2 + b^2 + bc .
\end{equation}
Assume that $|F(s) - F(t)| < 2|s-t|$ for all $s \neq t$, and let $P \neq Q$.
If $a = 0$, then $c = 0$ and \eqref{eq:abc} equals $b^2 > 0$.
If $a \neq 0$, then $|c| < 2|a|$ and $b^2 + bc \ge  -c^2/4 > -a^2$, so \eqref{eq:abc} is positive.
Conversely, if $|F(s_1) - F(s_2)| \ge  2|s_1 - s_2|$ for some $s_1 \neq s_2$, choose $t_1, t_2$ with $t_1 - t_2 = -c/2$.
Then $P \neq Q$ and \eqref{eq:abc} equals $a^2 - c^2/4 \le 0$, so $\theta_F$ is not strictly increasing.

(2) The condition in (1) is invariant under $F \mapsto -F$, so $\eta_F = \theta_{-F}$ is strictly increasing as well.
Since $\theta_F(\mathbb{R}^2) = \mathbb{R}^2$ and $\eta_F$ is an increasing left inverse of $\theta_F$, the uniqueness in \cite[Theorem 2]{Toth} gives $\theta_F^{(-1)} = \eta_F$, and likewise $\eta_F^{(-1)} = \theta_F$.
Let $x_i = (s_i, t_i)$ for $i = 1, \dots, m$, and put
\begin{equation*}
\bar{s} \coloneqq  \sum_{i=1}^m \lambda_i s_i, \qquad \bar{t} \coloneqq  \sum_{i=1}^m \lambda_i t_i, \qquad \delta \coloneqq  \sum_{i=1}^m \lambda_i F(s_i) - F(\bar{s}) .
\end{equation*}
Then $\sum_i \lambda_i \theta_F(x_i) = \left( \bar{s}, \ \bar{t} + \sum_i \lambda_i F(s_i) \right)$, and applying $\eta_F$ gives
\begin{equation}\label{eq:shearmean}
\mathcal{M}^{[m]}_{\theta_F,\lambda}(x_1, \dots, x_m) = (\bar{s}, \ \bar{t} + \delta), \qquad
\mathcal{M}^{[m]}_{\eta_F,\lambda}(x_1, \dots, x_m) = (\bar{s}, \ \bar{t} - \delta),
\end{equation}
the second formula by the same computation with $-F$ in place of $F$.
Adding the two formulas yields \eqref{eq:shearMS-weighted}.

(3) By \eqref{eq:shearmean} with $m = 2$ and $\lambda = (\tfrac{1}{2},\tfrac{1}{2})$, the mean $\mathcal{M}_{\theta_F}$ is the arithmetic mean if and only if
$F\big( \tfrac{1}{2}(s_1 + s_2) \big) = \tfrac{1}{2}\big( F(s_1) + F(s_2) \big)$ for all $s_1, s_2 \in \mathbb{R}$.
The condition in (1) implies that $F$ is continuous, and a continuous solution of this Jensen equation is affine; see \cite[Section 13.2]{Kuczma}.
Conversely, every affine $F$ satisfies the Jensen equation.
\end{proof}

Taking $m=2$ and $\lambda=(\tfrac{1}{2},\tfrac{1}{2})$ in Theorem \ref{thm:shear-weighted} gives Theorem \ref{thm:shear}.

\subsection{Properties of the shear solutions}

The condition in Theorem \ref{thm:shear-weighted}(1) holds for every Lipschitz function with Lipschitz constant less than $2$, and for every $C^1$ function $F$ with $|F'| < 2$ on $\mathbb{R}$, by the mean value theorem.
It is sharp. 
Indeed, for $F(s) = 2s$ the mapping $\theta_F$ is affine with a singular symmetric part of its Jacobian, and it is not strictly increasing.
For every non-affine $F$ satisfying the condition, for instance $F(s) = \sin s$ or $F(s) = |s|$, Theorem \ref{thm:shear-weighted} provides a solution of \eqref{eq:MSmon-weighted} on $C = \mathbb{R}^2$, valid for all $m \ge  2$ and all weights $\lambda$ simultaneously, in which neither of the two means is the arithmetic mean.
Since the set of admissible $F$ contains all $C^1$ functions with $|F'| < 2$, the family of these solutions is infinite-dimensional.

\begin{Rem}\label{rem:shear-features}
(1) The solutions of Theorem \ref{thm:shear} need not be differentiable: for $F(s) = |s|$ the generators $\theta_F$ and $\eta_F$ are not differentiable at the points of the line $\{x = 0\}$.
In dimension one, by contrast, every continuous strictly monotone solution of the Matkowski--Sut\^o equation is analytic, since the solutions are the affine pairs and the exponential pairs \cite{DP}.
Thus the regularity theory of \cite{DP} does not extend to dimension two in the framework of strictly increasing generators. \\
(2) The exponential pairs of dimension one solve the weighted equation \eqref{eq:MSmon-weighted} on $\mathbb{R}$ if and only if $m=2$ and $\lambda = (\tfrac{1}{2},\tfrac{1}{2})$.
Indeed, let $f(t) = \exp(pt)$ and $g(t) = -\exp(-pt)$ with $p>0$, and suppose that \eqref{eq:MSmon-weighted} holds for some $m\ge 2$ and $\lambda$.
Fix $i\in\{1,\dots,m\}$, set $t_i=s$ and $t_j=0$ for $j\neq i$, and let $s\to+\infty$.
Since $0<\lambda_i<1$, we obtain
\begin{align*}
\mathcal{M}^{[m]}_{f,\lambda}(t_1,\dots,t_m) +\mathcal{M}^{[m]}_{g,\lambda}(t_1,\dots,t_m) &= \frac{1}{p}\log\frac{\lambda_i \exp(ps)+1-\lambda_i}{\lambda_i \exp(-ps) +1-\lambda_i} \\
&= s+\frac{1}{p}\log\frac{\lambda_i}{1-\lambda_i}+o(1).
\end{align*}

The right-hand side of \eqref{eq:MSmon-weighted} is $2\lambda_i s$, so division by $s$ and passage to the limit give $\lambda_i=\tfrac{1}{2}$.
This holds for every $i$, and $\sum_i\lambda_i=1$ therefore implies $m=2$.
Conversely, for $m=2$ and equal weights, the exponential pairs solve the equation by \eqref{eq:MS1d}.
The shear pairs, on the other hand, solve \eqref{eq:MSmon-weighted} for every $m$ and every $\lambda$, because the deviation $\delta$ in \eqref{eq:shearmean} changes sign under $F \mapsto -F$ regardless of the weights. \\
(3) Identifying $\mathbb{R}^2$ with $\mathbb{C}$, the generator $\theta_F$ is the map $z \mapsto z + iF(\operatorname{Re} z)$, which is entire if and only if $F$ is constant on $\mathbb{R}$, since a real-valued entire function is constant.
Hence for nonconstant $F$ the shear generators are not holomorphic on $\mathbb{C}$, in accordance with Theorem \ref{thm:main}. \\
(4) The construction extends verbatim to $\mathbb{R}^{d_1} \times \mathbb{R}^{d_2}$: for $F \colon \mathbb{R}^{d_1} \to \mathbb{R}^{d_2}$ with $|F(a) - F(b)| < 2|a - b|$ for $a \neq b$, the mappings $(x', x'') \mapsto (x', x'' \pm F(x'))$ are strictly increasing and satisfy \eqref{eq:MSmon-weighted} on $\mathbb{R}^{d_1 + d_2}$, with the same proof.
For a scalar function $F$ satisfying Theorem \ref{thm:shear}(1), the same argument applies to a real Hilbert space $H$ of dimension at least two: writing $H=\mathbb{R}^2\oplus H_0$ as an orthogonal direct sum, extend $\theta_F$ and $\eta_F$ by the identity on $H_0$.
The monotonicity calculation acquires the additional term $\|u-v\|^2$ for $u,v\in H_0$, and the $H_0$-component of each generated mean is the weighted arithmetic mean, so the conclusions remain valid.
The two-dimensional construction also applies on the closed cylinders $C = I \times \mathbb{R}$ with $I \subset \mathbb{R}$ a closed interval and $F \colon I \to \mathbb{R}$.
\end{Rem}

\section{Gauss composition of the solution pairs}\label{sec:gauss}

Dar\'oczy and P\'ales \cite[Section~1.2]{DP} study the simultaneous iteration of two means on a real interval and define their Gauss composition as the common limit. 
Theorem~1.5 of \cite{DP} establishes convergence for two continuous means when at least one is strict. 
We now show that the solution pairs considered in this paper also have a Gauss composition, and that it equals the arithmetic mean.

Let $C$ be a nonempty convex subset of $\mathbb R^2$ (identified with
$\mathbb C$ in the holomorphic setting), and let $M_1,M_2:C\times C\to C$ satisfy
$M_j(x,x) = x$ for $j=1,2$. Starting from $(x_0,y_0)\in C\times C$, define
\begin{equation}\label{eq:gauss-iteration}
  x_{n+1}=M_1(x_n,y_n),\qquad  y_{n+1}=M_2(x_n,y_n),\qquad n\ge 0.
\end{equation}
When both sequences $(x_n)_n$ and $(y_n)_n$ converge to the same point of $C$ for every initial pair,
we call the resulting map $(x_0,y_0) \mapsto \lim_{n\to\infty} x_n = \lim_{n \to \infty} y_n$
the Gauss composition of $M_1$ and $M_2$ and denote it by $M_1\otimes M_2$.
If the pair satisfies the Matkowski--Sut\^o equation, then
\begin{equation}\label{eq:gauss-center}
  x_n+y_n=x_0+y_0=2c,\qquad
  c\coloneqq\frac{x_0+y_0}{2}.
\end{equation}
Thus any common limit must equal $c$. The following propositions establish
the existence of this limit for the holomorphic and shear solutions.

\begin{Prop}[Holomorphic solution pairs]\label{prop:gauss-holomorphic}
Let $U\subset\mathbb C$ be a convex domain and let $p\in P_U$.
For $M_1=A_{\chi_p}$ and $M_2=A_{\chi_{-p}}$, the iteration \eqref{eq:gauss-iteration} satisfies
\begin{equation*}
  \lim_{n\to\infty} x_n = \lim_{n\to\infty} y_n  =\frac{x_0+y_0}{2}  \qquad ((x_0,y_0)\in U\times U).
\end{equation*}
The same holds for every solution $(A_{\varphi_1}, A_{\varphi_2})$ of \eqref{eq:MS}.
\end{Prop}

\begin{proof}
The case $p=0$ is immediate, so assume $p \ne 0$. 
With $c$ as in \eqref{eq:gauss-center}, let
\begin{equation*}
  q_n\coloneqq \exp(p(x_n-c)).
\end{equation*}
Since $p \in P_U$, both means take values in $U$ and satisfy the Matkowski--Sut\^o equation. 
Hence the iteration is well defined for every $n\ge0$, and \eqref{eq:gauss-center} holds.

Consider the convex set $W\coloneqq \exp(-pc) \chi_p(U)$. 
Then $W$ does not contain $0$. 
Since $c, x_n, y_n \in U$ and $y_n - c = -(x_n - c)$, $W$ contains $1$, $q_n$, and $q_n^{-1}$.
By convexity, $W$ also contains
\[  \frac{q_n+|q_n|^2q_n^{-1}}{1+|q_n|^2} = \frac{2\operatorname{Re}q_n}{1+|q_n|^2}. \]
If $\operatorname{Re}q_n \le 0$, the segment joining $\frac{2\operatorname{Re}q_n}{1+|q_n|^2}$ to $1$ would contain $0$, contradicting $0 \notin W$. 
Therefore $\operatorname{Re}q_n > 0$. 

Let $r_n\coloneqq \frac{q_n-1}{q_n+1}$. 
Then $r_n$ is well defined, and $|q_n-1| < |q_n+1|$ gives $|r_n|<1$.

By the definition of $A_{\chi_p}$, $x_{n+1} = A_{\chi_p} (x_n, y_n)$ and $x_n + y_n = 2c$, 
we obtain that 
\[ \exp(px_{n+1}) = \frac{\exp(p x_n) + \exp(p y_n)}{2},  \qquad  q_{n+1} = \frac{q_n+q_n^{-1}}2. \]
Since 
\[  r_{n+1}  =\frac{q_n+q_n^{-1}-2}{q_n+q_n^{-1}+2}  =\left(\frac{q_n-1}{q_n+1}\right)^2 =r_n^2, \]
it holds that  
\begin{equation}\label{eq:gauss-squaring}
  r_{n+1} = r_n^2 \ \ (n \ge 0), \qquad |r_0|<1,
\end{equation}
and hence
\[ q_n=\frac{1+r_0^{\,2^n}}{1-r_0^{\,2^n}} \longrightarrow1, \quad n \to \infty. \]
Thus $\chi_p(x_n)\to\chi_p(c)$. 
The inverse of $\chi_p:U\to\chi_p(U)$ is continuous, so $x_n\to c$.
Since $y_n = 2c - x_n$, we also have $y_n\to c$.
Finally, affine changes of the generators leave the corresponding means unchanged, so the classification in Theorem~\ref{thm:main} gives the last assertion.
\end{proof}

\begin{Rem}\label{rem:gauss-quadratic}
For $p\neq0$, formula \eqref{eq:gauss-squaring} gives exact squaring in the coordinate $r_n$. 
For $x_0\neq y_0$, this is also quadratic convergence
in the original coordinates. Indeed, writing $d_n = x_n-c$, we see that $\exp(pd_{n+1}) = \cosh(pd_n)$. 
Since $d_n\to0$, Taylor expansion of the local inverse of
$d \mapsto \exp(pd)$ at the value $1$ gives
\begin{equation*}
  d_{n+1} = \frac{p}{2} d_n^2 + O(d_n^4).
\end{equation*}
Here the inverse is chosen to take the value $0$ at $1$; it exists because
the derivative of $d\mapsto \exp(pd)$ at $0$ is $p\neq0$.
\end{Rem}

\begin{Prop}[Shear solution pairs]\label{prop:gauss-shear}
Let $F$ satisfy the condition in Theorem~\ref{thm:shear}(1),
and let $\theta_F$ and $\eta_F$ be as in \eqref{eq:shear}.
For $M_1 = \mathcal M_{\theta_F}$ and $M_2 = \mathcal M_{\eta_F}$,
the iteration \eqref{eq:gauss-iteration} reaches its common limit after
at most two iterations:
\begin{equation*}
  x_n = y_n = \frac{x_0+y_0}{2}\qquad(n\geq2).
\end{equation*}
\end{Prop}

\begin{proof}
By \eqref{eq:shearmean} with $m=2$ and equal weights,
$x_1$ and $y_1$ have the same first coordinate.
The correction term in the same formula therefore vanishes at the next
iteration, giving $x_2 = y_2 = (x_1+y_1)/2 = (x_0+y_0)/2$, 
where the last equality follows from \eqref{eq:gauss-center}.
All subsequent iterates remain at this point.
\end{proof}

\begin{Rem}\label{rem:gauss-arbitrary-shear}
The two-step identity is algebraic. 
It holds for every function $F \colon \mathbb R\to\mathbb R$ if the two maps are defined using the ordinary inverses of $\theta_F$ and $\eta_F$. 
The condition on $F$ in Proposition~\ref{prop:gauss-shear} ensures that these generators belong to the strictly increasing framework of Theorem~\ref{thm:shear}.
\end{Rem}

For either class of solution pairs, let $C=U$ in the holomorphic case and $C=\mathbb R^2$ in the shear case. 
The arithmetic mean is also the unique continuous map $K:C\times C\to C$ satisfying
\[ K(x,x)=x,\quad K(M_1(x,y),M_2(x,y))=K(x,y).\]
Indeed, invariance gives $K(x_0,y_0)=K(x_n,y_n)$ for all $n$, and
the convergence proved above yields
\[  K(x_0,y_0)=K(c,c)=c=\frac{x_0+y_0}{2}.\]
Thus the interpretation of the Matkowski--Sut\^o equation through an invariant Gauss composition also holds for the two classes considered here. 
Nielsen and the author study Gauss composition for midpoint maps associated with an affine connection and its metric dual.
In that setting, invariance of the Riemannian midpoint implies local quadratic convergence under suitable regularity assumptions \cite[Theorem~3.2]{NielsenOkamura}.

\section{Concluding remarks and open problems}\label{sec:conclusion}

\subsection{Comparison of the two frameworks}\label{subsec:comparison}

The holomorphic and monotone frameworks impose different conditions on the generators.
Throughout this subsection, we identify $\mathbb{C}$ with $\mathbb{R}^2$ and use the Euclidean inner product and norm.
Let $\varphi$ be holomorphic on a convex domain $U$.
Integrating $\varphi'$ along the segment from $w$ to $z$ gives
\begin{equation}\label{eq:monotone}
\langle \varphi(z) - \varphi(w), z - w \rangle = |z-w|^2 \int_0^1 \operatorname{Re} \varphi'\big( w + t(z-w) \big)\, dt, \qquad z, w \in U ,
\end{equation}
so $\varphi$ satisfies the strict monotonicity inequality in Definition \ref{def:increasing} if $\operatorname{Re} \varphi' > 0$ on $U$, and $\varphi$ is not increasing if $\operatorname{Re} \varphi'$ takes a negative value.

A generator in $\mathcal{H}(U)$ need not be increasing, and this cannot always be remedied by an affine change of the generator as in Lemma \ref{lem:affine}, which does not change the mean.
For instance, $\varphi(z) \coloneqq  z/(1-z)$ belongs to $\mathcal{H}(U)$ for the unit disc $U$, since it maps $U$ onto the half-plane $\{ w : \operatorname{Re} w > -1/2 \}$, and $\varphi'(z) = (1-z)^{-2}$ has an argument filling the interval $(-\pi, \pi)$ as $z$ ranges over $U$; hence $\operatorname{Re}(a \varphi')$ changes sign on $U$ for every $a \in \mathbb{C} \setminus \{0\}$, and no $a\varphi + b$ is increasing or decreasing.
Even a real-linear change $L \circ \varphi + b$ with $L \in \mathrm{GL}(2,\mathbb{R})$ does not help: the Jacobian of $\varphi$ at $z$ is $|\varphi'(z)|$ times the rotation by $\arg \varphi'(z)$, and if $L \circ \varphi$ were increasing, the symmetric part of $L R_{\theta}$ would be positive semidefinite for every rotation $R_{\theta}$ with $\theta \in (-\pi,\pi)$; applied to $\theta$ and $\theta + \pi$ this forces $L R_{\theta}$ to be antisymmetric for all $\theta \in (-\pi, 0)$, which is impossible for an invertible $L$.

Conversely, a holomorphic map satisfying the strict monotonicity inequality in Definition \ref{def:increasing} need not belong to $\mathcal{H}(U)$: on the half disc $\{ z : |z| < 1, \ \operatorname{Re} z > 0 \}$ the map $z \mapsto z^2$ satisfies $\operatorname{Re}(2z) > 0$, hence the inequality by \eqref{eq:monotone}, but its image, the unit disc slit along $(-1,0]$, is not convex. 
As shown in \cite[Example 1]{Toth}, this map admits no extended monotone left inverse; the half disc is not closed, so \cite[Theorem 2]{Toth} does not apply.

The two settings also differ in the domain, which is closed in Definition \ref{def:increasing} and open in Definition \ref{def:H}, and in the equivalence of generators, which is complex affine in the holomorphic setting.

On $C = \mathbb{R}^2$, \cite[Theorem 4]{Toth} shows that the only strictly increasing generators of the arithmetic mean are the affine ones. 
This is consistent with $P_{\mathbb{C}} = \{0\}$ (Proposition \ref{prop:basic}).

In contrast to general elements of $\mathcal{H}(U)$, the solutions of \eqref{eq:MS} admit increasing representatives.
If $p \in P_U \setminus \{0\}$, then $\chi_p(U)$ is a convex domain not containing $0$, hence it is contained in an open half-plane $\{ w : \operatorname{Re}\left(\exp(-i\alpha) w\right) > 0 \}$ for some $\alpha \in \mathbb{R}$.
The generator $f \coloneqq  \exp(-i\alpha) p^{-1} \chi_p$ satisfies $A_f = A_{\chi_p}$ by Lemma \ref{lem:affine} and $\operatorname{Re} f' = \operatorname{Re}\left( \exp(-i\alpha) \exp(pz) \right) > 0$ on $U$, so $f$ is strictly increasing by \eqref{eq:monotone}; the same applies to $\chi_{-p}$.
Thus, up to affine changes of the generators, the exponential pairs of Theorem \ref{thm:main} satisfy the strict monotonicity inequality in Definition \ref{def:increasing}, although general elements of $\mathcal{H}(U)$ need not satisfy it.
The shear generators of Theorem \ref{thm:shear} with nonconstant $F$ are not holomorphic on $\mathbb{C}$ (Remark \ref{rem:shear-features}), 
so the rigidity expressed by Theorem \ref{thm:main} is a consequence of holomorphy rather than of monotonicity.

\subsection{Open problems}\label{subsec:open}

The holomorphic classification and the shear examples raise two related questions about the two-variable Matkowski--Sut\^o equation with equal weights.
These questions may be considered on convex subsets of real Hilbert spaces, possibly of infinite dimension, and on Banach spaces with an appropriate notion of monotonicity.
Affine equivalence below means composition of each generator with an invertible continuous affine transformation of its values.

\begin{Que}[Additional assumptions for classification]\label{que:classification}
What structural assumptions on the ambient space, the domain, and the generators, in addition to affine equivalence to strictly increasing mappings, yield an analogue of the affine/exponential classification?
In general ambient spaces, identifying suitable analogues of the exponential pairs $(\chi_p,\chi_{-p})$ in \eqref{eq:chi} is itself part of the problem.
\end{Que}

Section~\ref{sec:P} shows that the admissible exponents in the holomorphic setting depend on the domain, while Remark~\ref{rem:shear-features}(4) gives shear solutions on closed cylinders.
This leads to the following question.

\begin{Que}[Dependence on the domain]\label{que:domain}
How does the classification depend on the geometry of the domain?
In particular, how does it change when the generators are defined on a proper convex subset of the ambient space?
\end{Que}

\vspace{1pc}

\noindent{\bf Use of generative AI} \ During the preparation of this work, the author used Claude and ChatGPT to assist with exploration, structuring, drafting, and refining the text. 
After using these tools, the author reviewed and edited the content as needed. 
The author takes full responsibility for the content and integrity of the publication. 

\bibliographystyle{amsplain}
\bibliography{ms2d}

\end{document}